\documentclass[12pt,leqno]{amsart}
\usepackage[latin1]{inputenc}
\usepackage{color}
\usepackage{amssymb,amsmath,amsthm,amscd,amsfonts}
\usepackage{enumerate}
\usepackage[shortlabels]{enumitem}
\usepackage[all]{xy}
\usepackage{hyperref,cleveref,graphics,mathrsfs}

\numberwithin{equation}{section}
\def\hangbox to #1 #2{\vskip3pt\hangindent #1\noindent \hbox to #1{#2}$\!\!$}

\usepackage{hyperref,cleveref,amssymb,mathtools}

\theoremstyle{plain}
\newtheorem{theorem}{Theorem}[section]

\newtheorem{problem}[theorem]{Problem}

\theoremstyle{definition}
\newtheorem{remark}[theorem]{Remark}

\newtheorem{definition}[theorem]{Definition}

\DeclareSymbolFont{bbold}{U}{bbold}{m}{n}
\DeclareSymbolFontAlphabet{\mathbbold}{bbold}

\def\N{{\mathbb N}}
\def\R{{\mathbb R}}

\def\sfrac#1#2{\kern.1em\raise.5ex\hbox{$#1$}
        \kern-.1em/\kern-.05em\lower.25ex\hbox{$#2$}}

\def\vp{\varepsilon}
\def\dim{\operatorname{dim}}

\newcommand{\sign}{\text{\rm sign}}

\newcommand{\fw}{\text{\fw}}

\newcommand{\be}{\begin{equation}}
\newcommand{\ee}{\end{equation}}

\title[The Cahill-Casazza-Daubechies problem on H\"older stable phase retrieval]{The Cahill-Casazza-Daubechies problem on H\"older stable phase retrieval}
\author{Daniel\ Freeman}
\address{Department of Mathematics and Statistics\\
St Louis University\\
St Louis, MO   USA} \email{daniel.freeman@slu.edu}
\author{Mitchell~A.\ Taylor}
\address{Department of Mathematics\\
ETH Z\"urich, Ramistrasse 101, 8092 Z\"urich, Switzerland.
} \email{mitchell.taylor@math.ethz.ch}



\thanks{ The first author was supported by the National Science Foundation grant DMS-2154931. }

\begin{document}
\begin{abstract}
Phase retrieval using a frame for a finite-dimensional Hilbert space is known to always be Lipschitz stable.  However, phase retrieval using a frame or a continuous frame for an infinite-dimensional Hilbert space is always unstable.  In order to bridge the gap between the finite and infinite dimensional phenomena,
Cahill-Casazza-Daubechies (\textit{Trans.~Amer.~Math.~Soc.}~2016) gave a construction of a family of nonlinear subsets of an infinite-dimensional Hilbert space where phase retrieval could be performed with a H\"older stability estimate.  They then posed the question of whether these subsets satisfied Lipschitz stable phase retrieval.  We solve this problem both by giving examples which fail Lipschitz stability and by giving examples which satisfy Lipschitz stability.

\end{abstract}

\maketitle
\allowdisplaybreaks

\section{Phase retrieval and H\"older stability}
A \emph{frame} for a separable Hilbert space $\mathcal{H}$ is a sequence of vectors $\Phi=(\varphi_n)_{n\in I}$ in $\mathcal{H}$ for which there exists uniform constants $A,B>0$ satisfying
\begin{equation}\label{E:frame inequality}
    A\|f\|^2\leq \sum_{n\in 
    I}|\langle f,\varphi_n\rangle|^2\leq B\|f\|^2\hspace{1cm}\textrm{ for all }f\in \mathcal{H}.
\end{equation}
The {\em analysis operator} of the frame $\Phi$ is the map $T_\Phi:\mathcal{H}\rightarrow\ell_2(I)$ given by
$$T_\Phi(f)=(\langle f,\varphi_n\rangle)_{n\in I}.$$
By the frame inequality \eqref{E:frame inequality}, $T_\Phi$ is  an isomorphic embedding of $\mathcal{H}$ into $\ell_2(I)$.    The problem of \emph{phase retrieval with a frame} consists of recovering an unknown function $f\in \mathcal{H}$ from the set of intensity measurements $$|T_\Phi(f)|=\left(|\langle f,\varphi_n\rangle |\right)_{n\in I}\in \ell_2(I).$$ 
Since $|T_\Phi(\alpha f)|=|T_\Phi(f)|$ for all unimodular scalars $\alpha$, $|T_\Phi(f)|$ cannot distinguish $f$ from $\alpha f$. For this reason,  we define the equivalence relation $\sim$ on $\mathcal{H}$ by declaring that $f\sim g$ if $f=\alpha g$ for some unimodular scalar $\alpha$.  We  say that the frame \emph{$\Phi$ does phase retrieval} if  $|T_\Phi(f)|$ uniquely determines $f\in \mathcal{H}/\sim$.  In other words, the frame $\Phi$ does phase retrieval if the nonlinear mapping $$\mathcal{A}_\Phi: \mathcal{H}/\sim\to \ell_2(I),\hspace{5mm}\mathcal{A}_\Phi(f):=\left|T_\Phi(f)\right|$$
is injective. Phase retrieval problems arise in many applications, including coherent diffraction imaging \cite{miao1999extending}, transmission electron microscopy \cite{kirkland1998advanced} and speech recognition \cite{claudio1999speech}.
  \medskip

To ensure that the solution to a nonlinear inverse problem is reliable, it is imperative to understand the stability properties of the recovery map. For phase retrieval, it is known that the recovery map $\mathcal{A}_\Phi^{-1}$ is continuous, whenever it exists \cite{MR3656501}. However,  in practice, one needs quantitative stability estimates. We say that a frame $\Phi$ of a Hilbert space $\mathcal{H}$ does \emph{stable phase retrieval} if there exists a constant $C\geq 1$ such that for all $f,g\in \mathcal{H}$,
\begin{equation}\label{0.4}
    \inf_{|\alpha|=1}\|f-\alpha g\|_\mathcal{H}\leq C\|\mathcal{A}_\Phi(f)-\mathcal{A}_\Phi(g)\|_{\ell_2(I)}.
\end{equation}
Note that $\inf_{|\alpha|=1}\|f-\alpha g\|_\mathcal{H}$ is the distance between $[f]_\sim$ and $[g]_\sim$ in the quotient metric of $H/\!\sim$.  With this in mind, the inequality \eqref{0.4} states that the recovery map $\mathcal{A}_\Phi^{-1}$ exists and is $C$-Lipschitz continuous. If the Hilbert space $\mathcal{H}$ is finite-dimensional, then  any frame for $\mathcal{H}$ which does phase retrieval must do it stably \cite{MR3202304,MR3554699}. However, in every known explicit construction of a frame consisting of a number of vectors proportional to the dimension of the Hilbert space, the associated stability constant for phase retrieval increases to infinity as the dimension increases. Moreover, although a generic frame of $2n-1$ vectors for real $\ell_2^n$ does phase retrieval \cite{MR2224902}, Balan and Wang \cite{MR3323113} proved that all such solutions to the phase retrieval problem exhibit severe instabilities as the dimension $n$ tends to infinity. On the other hand,  many distributions for random frames yield dimension-independent stability bounds with high probability if one allows for a number of vectors bounded by a constant multiple above the optimal algebraic solutions \cite{MR3260258, MR3069958,MR3175089,MR3746047,KS}.
 \medskip

When the Hilbert space $\mathcal{H}$ is infinite-dimensional, the situation is entirely different. Indeed, although many important discrete and continuous frames for infinite-dimensional spaces perform phase retrieval, \emph{no} frame  for an infinite-dimensional $\mathcal{H}$ can perform \emph{stable} phase retrieval \cite{MR3554699}. The reason for this is that in any infinite-dimensional linear subspace $V$ of $\ell_2$  one can use a ``gliding hump" argument to find, for every $\epsilon>0$, normalized vectors $x=(a_n)$ and $y=(b_n)$ in $V$ so that
\begin{equation}\label{gliding hump}
    x\approx_\epsilon (a_1,\dots, a_N,0,0,\dots), \hspace{5mm}     y\approx_\epsilon (0,\dots, 0,b_{N+1},b_{N+2},\dots).
\end{equation}
One then observes that the vectors $x+y, \ x-y\in V$  are far from multiples of each other, yet satisfy $|x+y|\approx_\epsilon |x-y|$. In particular, taking $V:=T_\Phi(\mathcal{H})\subseteq \ell_2$, setting $x=T_\Phi(f), \ y=T_\Phi(g)$ and using the fact that $T_\Phi$ is an isomorphic embedding, one deduces a lack of stability in the phase recovery process for the frame $\Phi$.  A similar argument is given in \cite{MR3656501} to prove that no continuous frame for an infinite-dimensional Hilbert space does stable phase retrieval.
 \medskip

The above results suggest a dichotomy between finite and infinite dimensional phase retrieval problems. As phase retrieval in infinite dimensions is an important problem, there have been multiple different methods of bridging the gap to obtain some notion of stability in infinite dimensions.  One method is to use a different ambiguity than just considering two vectors as equivalent if they are equal up to a global phase.  In \cite{MR3989716,MR4298599,MR3977121,MR4188345,grohs2021stable}, circumstances are given where it is possible to partition the domain of a function where it is uniformly large into disconnected sets so that the function $f$ may be stably recovered from $|f|$ up to the ambiguity of having a different phase on each disconnected piece.  This is good enough for some applications, as, for example, a piece of audio that includes two parts $f$ and $g$ which are separated by an interval of silence sounds the same as if you changed the relative sign between $f$ and $g$.  
In \cite[Theorem 2.7]{MR3554699}, Cahill, Casazza, and Daubechies presented a different method of bridging the gap between finite and infinite dimensions by establishing H\"older stability of the phase recovery map for certain nonlinear subsets of $\ell_2$ consisting of functions which are ``well-approximated" by a sequence of finite-dimensional linear subspaces doing phase retrieval. More precisely, they proved the following theorem.
\begin{theorem}[Theorem 2.7 of \cite{MR3554699}]\label{CCD Thm}
Let $\mathcal{H}$ be an infinite-dimensional separable Hilbert space and let $V_1\subseteq V_2\subseteq V_3...$ be a nested sequence of finite-dimensional subspaces of $\mathcal{H}$. Let $\Phi=(\varphi_n)_{n\in \mathbb{N}}$ be a frame for $\mathcal{H}$ with frame bounds $0<A\leq B<\infty$ and suppose that there is an increasing function $G(m)$ with $\lim_{m\to\infty}G(m)=\infty$ such that for every $m$,
\begin{equation}\label{E:CCD1}
\inf_{|\alpha|=1}\|f-\alpha g\|\leq G(m)\|\mathcal{A}_\Phi(f)-\mathcal{A}_\Phi(g)\|\hspace{1cm}\textrm{ for all }f,g\in V_m.
\end{equation}
For $\gamma>1$ and $R>0$ define 
\begin{equation}\label{E:CCD2}
\mathcal{B}_\gamma(R)=\{f\in \mathcal{H} : \|f-P_mf\|\leq G(m+1)^{-\gamma}R\|f\|\ \text{for every}\ m\in \mathbb{N}\},
\end{equation}
where $P_m$ denotes the orthogonal projection onto $V_m$. Then there exists a constant $C>0$ depending only on $B$, $R$, $\gamma$ and $G(1)$ such that for all $f,g\in \mathcal{B}_\gamma(R)$,
\begin{equation}\label{conclusion}
    \inf_{|\alpha|=1}\|f-\alpha g\|_\mathcal{H}\leq C\left(\|f\|_\mathcal{H}+\|g\|_\mathcal{H}\right)^\frac{1}{\gamma}\|\mathcal{A}_\Phi(f)-\mathcal{A}_\Phi(g)\|_{\ell_2}^{\frac{\gamma-1}{\gamma}}.
\end{equation}
\end{theorem}
 Note that the condition \eqref{conclusion} states that the phase recovery map associated to $\mathcal{B}_\gamma(R)$ is H\"older continuous on the unit ball, whereas \eqref{0.4} requires that this map be Lipschitz continuous. In \cite[Remark 2.8]{MR3554699}, Cahill, Casazza and Daubechies posed the following question.
 \begin{problem}\label{Problem 1}
     Do the sets $\mathcal{B}_\gamma(R)$ in \Cref{CCD Thm} do Lipschitz stable phase retrieval, i.e., does there exist a constant $C\geq 1$ such that for all $f,g\in \mathcal{B}_\gamma(R)$,
     \begin{equation*}
    \inf_{|\alpha|=1}\|f-\alpha g\|_\mathcal{H}\leq C\|\mathcal{A}_\Phi(f)-\mathcal{A}_\Phi(g)\|_{\ell_2}?
\end{equation*}
 \end{problem}
The purpose of this paper is to answer this question. In Section \ref{S:Counter} we produce a counterexample by providing a frame $\Phi$ which does phase retrieval for an infinite dimensional Hilbert space $\mathcal{H}$ and a sequence of finite dimensional subspaces $V_1\subseteq V_2\subseteq...$ such that for all $\gamma>1$, all choices of $(G(m))_{m=1}^\infty$, and all $R>0$ we have that $\Phi$ fails to do Lipschitz stable phase retrieval on $\mathcal{B}_\gamma(R)$.  On the other hand, in Sections \ref{S:L} and \ref{S:M} we provide constructions in both real and complex infinite-dimensional Hilbert spaces where phase retrieval is Lipschitz stable on the set $\mathcal{B}_\gamma(R)$. Thus, the question of Lipschitz stability on a subset is much more delicate than for subspaces and it greatly depends on the relationship between the frame $\Phi$ and the sequence of subspaces $(V_n)_{n=1}^\infty$.
\medskip

 In typical applications of phase retrieval, the model is to recover  a vector $f$ in an infinite-dimensional Hilbert space from the magnitudes of its frame coefficients.  However, we are only realistically able to consider finitely many measurements and instead of reconstructing the full vector $f$, we are satisfied with reconstructing an approximation $P_{H_0} f$ where $P_{H_0}$ is the orthogonal projection onto a finite dimensional subspace of $\mathcal{H}$.  Phase retrieval is necessarily stable on this subspace $H_0$, but if we wish for $P_{H_0} f$ to be a very good approximation to $f$ then we must choose $H_0$ to be a very large subspace which will result in the stability constant for phase retrieval on $H_0$ being exceedingly large. 
 This stability problem can be fixed if we know that $f\in \mathcal{B}_\gamma(R)$ and that $\Phi$ does Lipschitz stable phase retrieval on $\mathcal{B}_\gamma(R)$.  In this case, we may choose the subspace $H_0$ as large as we want while maintaining a uniform bound for the stability of phase retrieval.  
 \medskip

Before discussing our examples and counterexample to Problem \ref{Problem 1}, we feel that it is instructive to first compare the above formulation of phase retrieval for frames with a more recent approach from \cite{calderbank2022stable,christ2022examples,FOTP,thesis2} which permits stable recovery for infinite-dimensional subspaces. For this, we note that, up to relabeling the constant $C$, the inequality \eqref{0.4} is equivalent to the inequality 
\begin{equation}\label{new ineq}
    \inf_{|\alpha|=1}\|T_\Phi(f)-\alpha T_\Phi(g)\|_{\ell_2(I)}\leq C \| |T_\Phi(f)|-|T_\Phi(g)|\|_{\ell_2(I)}.
\end{equation}
Since the operator $T_\Phi$ appears on both sides of the inequality \eqref{new ineq}, we may relabel $f\leftrightarrow T_\Phi(f)$ and $g\leftrightarrow T_\Phi(g)$ to see that a frame $\Phi$ does stable phase retrieval if and only if there exists a constant $C\geq 1$ such that
\begin{equation}\label{new ineq2}
    \inf_{|\alpha|=1}\|f-\alpha g\|_{\ell_2(I)}\leq C \| |f|-|g|\|_{\ell_2(I)} \hspace{5mm} \text{for all} \ f,g\in T_{\Phi}(\mathcal{H})\subseteq \ell_2(I).
\end{equation}
The importance of the reformulated inequality \eqref{new ineq2} is that it makes no reference to the operator $T_\Phi$; instead, the operator $T_\Phi$ has been encoded into the linear subspace $T_{\Phi}(\mathcal{H})\subseteq \ell_2(I)$  for which the inequality in \eqref{new ineq2} is required to hold. The following definition presents a vast generalization of stable phase retrieval for frames which also encompasses the situation in \eqref{new ineq2}.

\begin{definition}\label{def of pr}
    Let $(\Omega,\Sigma,\mu)$ be a measure space and let $0<\sigma\leq 1$. A subset $V\subseteq L_2(\mu)$ is said to do $\sigma$-\textit{H\"older stable phase retrieval} if there exists a constant $C\geq 1$ such that for all $f,g\in V$, 
\begin{equation}\label{HSPR Def}
\inf_{|\alpha|=1}\|f-\alpha g\|_{L_2}\leq C\||f|-|g|\|_{L_2}^\sigma\left(\|f\|_{L_2}+\|g\|_{L_2}\right)^{1-\sigma}.
\end{equation}
When \eqref{HSPR Def} holds with $\sigma=1$, we say that $V\subseteq L_2(\mu)$ does \emph{(Lipschitz) stable phase retrieval}.
\end{definition}
As demonstrated by \eqref{new ineq2}, phase retrieval problems for a frame $\Phi$ are equivalent to phase retrieval problems for the linear subspace $T_\Phi(\mathcal{H})\subseteq \ell_2(I).$ However, inequalities of the form \eqref{HSPR Def}  also arise when considering other operators $T:\mathcal{H}\to L_2(\mu)$ which embed a Hilbert space $\mathcal{H}$ into a function space $L_2(\mu)$. Indeed, there are many physical situations where one must recover some vector $f$ up to global phase from measurements of the form $|Tf|$, where $T$ is not the analysis operator of a discrete frame. The most classical examples are the Fourier and Pauli phase retrieval problems, which appear in crystallography \cite{MR4094471,MR3674475} and quantum mechanics \cite{MR4610225,pauli1933allgemeinen}, respectively. Other important examples are the STFT and wavelet phase retrieval problems \cite{MR4537580,MR3935477, MR4404785,MR4828017}, and a physically interesting example arising from a nonlinear operator $T$ is the radar ambiguity problem \cite{MR2362409,MR1700086}. As shown in \cite{calderbank2022stable,christ2022examples,FOTP,thesis2}, these problems are all special cases of \Cref{def of pr}, and for many of these problems it is possible to achieve stability in infinite dimensions. We also remark that \Cref{def of pr} makes sense with $L_2(\mu)$ replaced by $L_p(\mu)$ or even a general Banach lattice. Unified perspectives on nonlinear inverse problems involving lattice operations (e.g.~phase retrieval, declipping and ReLU recovery) or more general piecewise-linear maps (such as those arising in unlimited sampling) can be found in \cite{abdalla2025sharp,freeman2025optimal}. Physical and mathematical evidence that the prior $V$ in \Cref{def of pr} may in general only be a subset and not a subspace can be found in \cite{aslan2025ptygenography,bendory2023phase}.
 \medskip

In the language of \Cref{def of pr}, \Cref{CCD Thm}
states that the phase and scaling invariant subset $T_\Phi(\mathcal{B}_\gamma(R))\subseteq \ell_2$ does $\frac{\gamma-1}{\gamma}$-H\"older stable phase retrieval.   In \cite{FOTP}, it is shown that if $V$ is a \emph{linear subspace} of $L_2(\mu)$, then H\"older stable phase retrieval is equivalent to Lipschitz stable phase retrieval. More precisely, we have the following theorem.
\begin{theorem}[Corollary 3.12 of \cite{FOTP}]\label{H implies Lip}
    Let $V$ be a linear subspace of a real or complex Hilbert space $L_2(\mu)$. If $V$ does $\sigma$-H\"older stable phase retrieval with constant $C$ then $V$ does stable phase retrieval with constant $(4C)^\frac{1}{\sigma}$.
\end{theorem}
The proof of \Cref{H implies Lip} is based on classifying the extremizers of the phase retrieval inequality. As observed in \cite{localphase} and \cite{FOTP}, instabilities are maximized on orthogonal vectors. 
\begin{theorem}[\cite{localphase,FOTP}]\label{orth red}
    Given any linearly independent $f,g\in L_2(\mu)$ there exist orthogonal vectors $f',g'\in span\{f,g\}$ such that
    \begin{equation*}
        \inf_{|\alpha|=1}\|f-\alpha g\|=\inf_{|\alpha|=1}\|f'-\alpha g'\|
        \end{equation*}
    and 
    \begin{equation*}
        \| |f'|-|g'|\|\leq \| |f|-|g|\|.
    \end{equation*}
\end{theorem}
As a corollary of \Cref{orth red}, to prove that a linear subspace $V\subseteq L_2(\mu)$ does stable phase retrieval, it suffices to check that \eqref{HSPR Def} holds for $f,g\in V$ satisfying $\|f\|=1,$ $\|g\|\leq 1$ and $f\perp g.$ In this case, $\|f-\alpha g\|=(\|f\|^2+\|g\|^2)^{\frac{1}{2}}\in [1,\sqrt{2}]$ is independent of the choice of $\alpha$ and is uniformly bounded away from zero. Having reduced to checking \eqref{HSPR Def} for vectors $f$ and $g$ which  are  ``well-separated" in the quotient metric, it is easy to see that H\"older and Lipschitz stability are equivalent on linear subspaces.  This fact was critically used in \cite{christ2022examples} to construct the first examples of infinite-dimensional linear subspaces of complex $L_2(\mathbb{T})$ doing stable phase retrieval, as the argument in \cite{christ2022examples} only directly established H\"older stability.
 \medskip

In view of Theorems~\ref{H implies Lip} and \ref{orth red},  it would be natural to conjecture that the subsets $\mathcal{B}_\gamma(R)$ in \Cref{CCD Thm} do Lipschitz stable phase retrieval. However, the first objective of this article is to show that H\"older stability is the best one can achieve, in general, which answers the question posed by Cahill, Casazza and Daubechies in \cite[Remark 2.8]{MR3554699}. In particular,  the orthogonalization procedure we used to produce well-separated vectors witnessing instabilities in \Cref{orth red} need not leave the sets $\mathcal{B}_\gamma(R)$ invariant. Instead, we will carefully construct an example of a set $\mathcal{B}_\gamma(R)$ where we can create close vectors with even closer moduli at a rate consistent with ruling out some, but not all, H\"older stability exponents.

\section{A counterexample to Lipschitz stable phase retrieval}\label{S:Counter}
We now present our counterexample to the Cahill-Casazza-Daubechies problem.

\begin{theorem}\label{Main Theorem}
    Let $\mathcal{H}$ be an infinite-dimensional separable Hilbert space (real or complex).  There exists a sequence of nested finite-dimensional linear subspaces $(V_m)_{m=1}^\infty$ in $\mathcal{H}$ such that $\dim(V_m)=m$ for all $m\in\N$, a frame $\Phi$ of $\mathcal{H}$, and an increasing function $G(m)$ satisfying \eqref{E:CCD1} such that
for every $\gamma>1$, $R>0$, $\sigma>1+\gamma$ and $C>0$ there exists $f,g\in \mathcal{B}_\gamma(R)$ with
$$    \inf_{|\alpha|=1}\|f-\alpha g\|_\mathcal{H}> C\left(\|f\|_\mathcal{H}+\|g\|_\mathcal{H}\right)^\frac{1}{\sigma}\|\mathcal{A}_\Phi(f)-\mathcal{A}_\Phi(g)\|_{\ell_2}^{\frac{\sigma-1}{\sigma}}.
$$
In particular, the frame $\Phi$ fails to do Lipschitz stable phase retrieval on $\mathcal{B}_\gamma(R)$ and even fails to do $(\sigma-1)/\sigma$-H\"older stable phase retrieval for any $\sigma>1+\gamma$.
\end{theorem}

\begin{proof}
   Let $(e_n)_{n=1}^\infty$ be an orthonormal basis for $\mathcal{H}$ and consider $V_m=\text{span}\{e_1,\dots,e_m\}\subseteq \mathcal{H}$ for all $m\in\N$. We let the frame $\Phi$  consist of the union of $(e_n)_{n=1}^\infty$ and $(\frac{1}{2^n}x_{j,n})_{n\in \mathbb{N},\ 1\leq j\leq N_n}$ where for each $n\in \mathbb{N}$, $(x_{j,n})_{1\leq j\leq N_n}$ is a Parseval frame for $V_n$ which does $C$-stable phase retrieval for some uniform constant $C$. It is well-known that such Parseval frames $(x_{j,n})_{1\leq j\leq N_n}$ exist, and it is easy to check that $\Phi$ forms a frame for $\mathcal{H}$ with lower frame bound $1$ and upper frame bound at most $\frac{4}{3}$.
     \medskip

    We next note that since $(x_{j,n})_{1\leq j\leq N_n}$ is a Parseval frame doing $C$-stable phase retrieval for $V_m$ we have that for any $m\in \mathbb{N}$  and $f,g\in V_m$, 
    \begin{equation*}
        \begin{split}
            \inf_{|\alpha|=1}\|f-\alpha g\|_\mathcal{H}\leq 2^mC\|\mathcal{A}_\Phi(f)-\mathcal{A}_\Phi(g)\|.
        \end{split}
    \end{equation*}
    We may thus set $G(m)=2^mC.$
     \medskip

    We now fix $\gamma>1$, $R>0$ and let the subset $\mathcal{B}_\gamma(R)\subseteq\mathcal{H}$ be defined as in \eqref{E:CCD2}.  Thus, we have 
    \begin{equation*}
\mathcal{B}_\gamma(R)=\{f\in \mathcal{H} : \|f-P_nf\|\leq 2^{-(n+1)\gamma}C^{-\gamma}R\|f\|\ \text{for every}\ n\in \mathbb{N}\}.
\end{equation*}
      Given $m\in \mathbb{N}$, we let $x:=e_1+\frac{R}{2^{m\gamma}C^\gamma}e_m$ and $y:=e_1-\frac{R}{2^{m\gamma}C^\gamma}e_m$. Note that for $m$ sufficiently large, we have
    \begin{equation}\label{norm x-y}
        \inf_{|\alpha|=1}\|x-\alpha y\|=\|x-y\|=\frac{2R}{2^{m\gamma}C^\gamma}
    \end{equation}
and 
  \begin{equation*}
  \begin{cases}
    &\|x-P_nx\|=\frac{R}{2^{m\gamma}C^\gamma}, \ \text{if} \ n<m,
    \\  
    &\|x-P_nx\|=0, \ \ \ \ \ \ \   \text{if} \ n\geq m,
    \end{cases}
\end{equation*}
with similar identities for $y$ in place of $x$. Thus, $x,y\in \mathcal{B}_\gamma(R)$. However,
\begin{equation*}
    \begin{split}
        \|\mathcal{A}_\Phi(x)-\mathcal{A}_\Phi(y)\|^2&=\sum_{n=1}^\infty\big||\langle x,\varphi_n\rangle|-|\langle y,\varphi_n\rangle |\big|^2\\
        &=\sum_{n=m}^\infty \sum_{j=1}^{N_n} \left|\left|\langle x,\frac{x_{j,n}}{2^n}\rangle \right|-\left|\langle y,\frac{x_{j,n}}{2^n}\rangle\right|\right|^2
        \\
        &\leq\sum_{n=m}^\infty \frac{1}{4^n}\sum_{j=1}^{N_n} \left|\langle x-y,x_{j,n}\rangle\right|^2
        \\
        &\leq \sum_{n=m}^\infty\frac{1}{4^n}\|x-y\|^2\hspace{.5cm}\textrm{ as $(x_{j,n})_{j=1}^{N_n}$ has upper frame bound $1$}.
    \end{split}
\end{equation*}
Hence,
\begin{equation*}
    \|\mathcal{A}_\Phi(x)-\mathcal{A}_\Phi(y)\|\leq \frac{2}{\sqrt{3}}\cdot \frac{1}{2^m}\|x-y\|.
\end{equation*}
    Now suppose that an inequality of the form \eqref{conclusion} holds on $\mathcal{B}_\gamma(R)$,  i.e., for some $0< \mu\leq 1$ and all $f,g\in \mathcal{B}_\gamma(R)$ we have the inequality
    \begin{equation*}\label{conclusion2}
    \inf_{|\alpha|=1}\|f-\alpha g\|_\mathcal{H}\leq K\left(\|f\|_\mathcal{H}+\|g\|_\mathcal{H}\right)^{1-\mu}\|\mathcal{A}_\Phi(f)-\mathcal{A}_\Phi(g)\|_{\ell_2}^{\mu}.
\end{equation*}
Note that for all sufficiently large $m$ we have $\|x\|,\|y\|\leq 2$. Hence,
\begin{equation*}
    \begin{split}
    \|x-y\|\leq K 4^{1-\mu}\left[\frac{2}{\sqrt{3}}\cdot \frac{1}{2^m}\|x-y\|\right]^{\mu}.
    \end{split}
\end{equation*}
Sending $m\to\infty$ and using \eqref{norm x-y},  we obtain the restriction $\mu\leq \frac{\gamma}{1+\gamma}<1$. Thus, $\mathcal{B}_\gamma(R)$ cannot do Lipschitz stable phase retrieval or even $\mu$-H\"older stable phase retrieval if $\mu$ is close enough to one depending on $\gamma$. Setting $\mu=\frac{\sigma-1}{\sigma}$ completes the proof of \Cref{Main Theorem}.
\end{proof}

\section{Obtaining Lipschitz stable phase retrieval}\label{S:L}
Despite \Cref{Main Theorem}, we will prove that it is still possible for $\mathcal{B}_\gamma(R)$ to be ``large" and satisfy Lipschitz stable phase retrieval. For the case that $\mathcal{H}$ is a real Hilbert space, we are able to construct a frame $(\varphi_n)_{n=1}^\infty$ of $\mathcal{H}$ which is a small perturbation of an orthonormal basis and does stable phase retrieval on some set $\mathcal{B}$ which is of the same form as the sets constructed in \Cref{CCD Thm} (Theorem 2.7 in \cite{MR3554699}).  That is, we take a nested sequence of finite-dimensional linear subspaces $(V_m)_{m=1}^\infty$ and a rapidly decreasing sequence $\beta_n\searrow0$ and define $\mathcal{B}$ as the set of all vectors $f\in \mathcal{H}$ such that for all $m\in\N$, the distance between $f$ and $V_m$ is at most $\beta_{m+1}\|f\|_\mathcal{H}$.  
\begin{theorem}\label{T:R}
Let $(e_n)_{n=1}^\infty$ be an orthonormal basis for an infinite-dimensional real Hilbert space $\mathcal{H}$. For each $m\in\N$ let $V_m=\textrm{span}_{1\leq n\leq m}e_n$ and let $P_m:\mathcal{H}\rightarrow V_m$ be the orthogonal projection. Fix $0<\vp<8^{-1}$ and choose $\alpha_n\searrow 0$ such that $\sum_{n=2}^\infty \alpha_n<\vp/2$ and $\beta_n\searrow 0$ such that $\beta_n<2^{-1}\alpha_n$ for all $n\in\N$.   Define the subset
\begin{equation*}
\mathcal{B}=\{f\in \mathcal{H} : \|f-P_mf\|\leq \beta_{m+1}\|f\|\ \text{for every}\ m\in \mathbb{N}\}
\end{equation*} and consider the Riesz basis $(\varphi_n)_{n=1}^\infty$ of $\ell_2$ given by $\varphi_1=e_1$ and $\varphi_n=\alpha_n e_1+e_n$ for all $n\geq 2.$ 
Then for all $f,g\in \mathcal{B}$ we have
\begin{equation*}
\inf_{|\alpha|=1}\|f-\alpha g\|\leq (1-\vp)^{-1/2}\|\mathcal{A}_\Phi(f)-\mathcal{A}_\Phi(g)\|.
\end{equation*}
\end{theorem}
\begin{proof}
Let $f,g\in\mathcal{B}$ with basis expansions $f=\sum_{n=1}^\infty a_n e_n$ and $g=\sum_{n=1}^\infty b_n e_n$.  By the definition of $\mathcal{B}$, we have $a_1=0$
if and only if $f=0$. Thus, by scaling, we may assume that $a_1>0$ and $b_1>0$.
We have that
$$a_1=\|P_1 f\|\geq \|f\|-\|f-P_1f\|\geq (1-\beta_2)\|f\|\geq (1-4^{-1}\vp)\|f\|\geq 2^{-1}\|f\|$$
and for all $n\geq 2$ we have 
$$\beta_n \|f\|\geq \|f-P_{n-1}f\|\geq |a_n|.$$
Hence, $2\beta_n a_1\geq |a_n|$. As $\beta_n\leq 2^{-1}\alpha_n$, it follows that $\alpha_n a_1\geq |a_n|$.
Thus, for all $n\geq2$ we have that
$$|\langle f,\varphi_n\rangle|=|\alpha_n a_1+a_n|=\alpha_n a_1+a_n.
$$
Likewise, we have that $|\langle g,\varphi_n\rangle|=|\alpha_n b_1+b_n|=\alpha_n b_1+b_n
$ for all $n\geq 2$.
Hence, we may compute that
\begin{align*}
\|\mathcal{A}_\Phi(f)-\mathcal{A}_\Phi(g)\|^2&=\sum_{n=1}^\infty \big||\langle f,\varphi_n\rangle|-|\langle g,\varphi_n\rangle|\big|^2\\
&=|a_1-b_1|^2+\sum_{n=2}^\infty \big|(\alpha_n a_1+a_n)-(\alpha_n b_1+b_n)\big|^2\\
&\geq|a_1-b_1|^2+\sum_{n=2}^\infty \big| |a_n-b_n|-\alpha_n |a_1-b_1|\big|^2\\
&\geq |a_1-b_1|^2+\sum_{n=2}^\infty |a_n-b_n|^2-2|a_1-b_1|\alpha_n|a_n-b_n|\\
&= \|f-g\|^2-\sum_{n=2}^\infty 2|a_1-b_1|\alpha_n|a_n-b_n|\\
&\geq \|f-g\|^2-\sum_{n=2}^\infty 2\alpha_n\|f-g\|^2\\
&\geq \|f-g\|^2-\vp\|f-g\|^2.
\end{align*}
It follows that
$$\inf_{|\alpha|=1}\|f-\alpha g\|\leq \|f-g\|\leq (1-\vp)^{-1/2}\|\mathcal{A}_\Phi(f)-\mathcal{A}_\Phi(g)\|.
$$

\end{proof}

It is clear that in \Cref{T:R}, $(\varphi_n)_{n=1}^\infty$ does not do phase retrieval on $\mathcal{H}$ because no basis does phase retrieval for a Hilbert space.  However, we also have for each $m\geq 2$ that $(P_{V_m}\varphi_n)_{n=1}^\infty$ does not do phase retrieval on $V_m$.  This is in stark contrast to \cite[Theorem 2.7]{MR3554699} (\Cref{CCD Thm} in the current paper), where they require that $(P_{V_m}\varphi_n)_{n=1}^\infty$ does phase retrieval on $V_m$ and make explicit use of the stability constants in their proof.  Hence, we obtain Lipschitz stability with even weaker conditions than the hypothesis for \cite[Theorem 2.7]{MR3554699}.  Note that by including additional frame vectors in the construction in \Cref{T:R}, one can obtain a frame which does Lipschitz stable phase retrieval on $\mathcal{B}$ and does phase retrieval on each $V_m$, so satisfies the conditions in \cite[Theorem 2.7]{MR3554699}.
\medskip

We now prove that it is possible to obtain Lipschitz stability in the complex case, which is more delicate.
\begin{theorem}\label{T:C}
Let $(e_n)_{n=1}^\infty$ be an orthonormal basis for an infinite-dimensional complex Hilbert space $\mathcal{H}$. For each $m\in\N$ let $V_m=\textrm{span}_{1\leq n\leq m}e_n$ and let $P_m:\mathcal{H}\rightarrow V_m$ be the orthogonal projection. Choose $\alpha_n\searrow 0$ such that $\sum_{n=2}^\infty \alpha_n<1/200$ and let $\beta_n\searrow 0$ be such that $\beta_n<2^{-1}\alpha^2_n$ for all $n\in\N$.    Consider the frame $e_1\cup(\varphi_{n,1})_{n=2}^\infty\cup(\varphi_{n,i})_{n=2}^\infty$ of $\mathcal{H}$ given by $\varphi_{n,1}=\alpha_n e_1+e_n$ and $\varphi_{n,i}=\alpha_n e_1+ie_n$ for all $n\geq 2.$ Define the subset
\begin{equation*}
\mathcal{B}=\{f\in \mathcal{H} : \|f-P_mf\|\leq \beta_{m+1}\|f\|\ \text{for every}\ m\in \mathbb{N}\}.
\end{equation*}
Then for all $f,g\in \mathcal{B}$ we have 
\begin{equation*}
\inf_{|\alpha|=1}\|f-\alpha g\|\leq 5\|\mathcal{A}_\Phi(f)-\mathcal{A}_\Phi(g)\|.
\end{equation*}
\end{theorem}
\begin{proof}

Let $f,g\in\mathcal{B}$ with $f=(a_n)_{n=1}^\infty$ and $g=(b_n)_{n=1}^\infty$.  By the definition of $\mathcal{B}$, we have $a_1=0$ if and only if $f=0$. Thus, by scaling, we may assume that $a_1,b_1\in\R$ satisfy $a_1, b_1>0$ and that $\|f\|=1$ and $0\leq \|g\|\leq 1$. 
It is easy to see that
\begin{align*}
|\langle f,\varphi_{n,1}\rangle|^2&=|\alpha_n a_1+a_n|^2\\
&=(\alpha_n a_1+\Re a_n)^2+(\Im a_n)^2\\
&=\alpha^2_n a^2_1+2\alpha_n a_1\Re a_n+(\Re a_n)^2+(\Im a_n)^2\\
&=\alpha^2_n a^2_1+2\alpha_n a_1\Re a_n+|a_n|^2.
\end{align*}
Likewise, 
$$
|\langle f,\varphi_{n,i}\rangle|^2
=\alpha^2_n a^2_1-2\alpha_n a_1\Im a_n+|a_n|^2.
$$
For the case of $\varphi_{n,1}$ we have  the identity
\begin{equation}\label{E:diff}
\big||\langle f,\varphi_{n,1}\rangle|-|\langle g,\varphi_{n,1}\rangle|\big|=\big|(\alpha^2_n a^2_1+2\alpha_n a_1\Re a_n+|a_n|^2)^{1/2}-(\alpha^2_n b^2_1+2\alpha_n b_1\Re b_n+|b_n|^2)^{1/2}\big|.
\end{equation}
As in the proof of \Cref{T:R}, we have that  $1\geq a_1\geq 0$ and $\alpha_n a_1\geq |a_n|$. This implies the inequality $2\alpha_n\geq (\alpha^2_n a^2_1+2\alpha_n a_1\Re a_n+|a_n|^2)^{1/2}$. The corresponding inequality for $b_n$ likewise holds, and we may conclude that
\begin{equation}\label{E:alpha}
4\alpha_n\geq  (\alpha^2_n a^2_1+2\alpha_n a_1\Re a_n+|a_n|^2)^{1/2}+(\alpha^2_n b^2_1+2\alpha_n b_1\Re b_n+|b_n|^2)^{1/2}.
\end{equation}
By multiplying \eqref{E:diff} and \eqref{E:alpha}, we obtain
$$4\alpha_n\big||\langle f,\varphi_{n,1}\rangle|-|\langle g,\varphi_{n,1}\rangle|\big|\geq \big|\alpha^2_n a^2_1+2\alpha_n a_1\Re a_n+|a_n|^2-(\alpha^2_n b^2_1+2\alpha_n b_1\Re b_n+|b_n|^2)\big|.
$$
Thus, we have that
\begin{equation}\label{E:1}
16\big||\langle f,\varphi_{n,1}\rangle|-|\langle g,\varphi_{n,1}\rangle|\big|^2\geq \big|\alpha_n (a^2_1-b^2_1)+2(a_1\Re a_n-b_1\Re b_n)+\alpha_n^{-1}(|a_n|^2-|b_n|^2)\big|^2.
\end{equation}
Note that $0\leq a_1,b_1\leq 1$ and $0\leq |a_n|,|b_n|\leq 2^{-1}\alpha_n^2$ for all $n\geq 2$.
This allows us to derive the following three inequalities:
\begin{equation}\label{E:2}
|a_1^2-b_1^2|=(a_1+b_1)|a_1-b_1|\leq 2\|f-g\|,
\end{equation}
\begin{equation}\label{E:3}
\big||a_n|^2-|b_n|^2\big|=(|a_n|+|b_n|)\big||a_n|-|b_n|\big|\leq \alpha_n^2\|f-g\|,
\end{equation}
\begin{equation}\label{E:4}
|a_1\Re a_n-b_1\Re b_n|=|a_1(\Re a_n-\Re b_n)+(a_1-b_1)\Re b_n|\leq 2\|f-g\|.
\end{equation}
By \eqref{E:1}, \eqref{E:2}, \eqref{E:3} and \eqref{E:4} we deduce that 
\begin{align*}
16\big||\langle &f,\varphi_{n,1}\rangle|-|\langle g,\varphi_{n,1}\rangle|\big|^2
\geq 4|a_1\Re a_n-b_1\Re b_n|^2-4|a_1\Re a_n-b_1\Re b_n|\big(\alpha_n |a^2_1-b^2_1|+\alpha_n^{-1}\big||a_n|^2-|b_n|^2\big|\big)\\
&\geq 4|a_1\Re a_n-b_1\Re b_n|^2
-8\|f-g\|\big(2\alpha_n \|f-g\|+\alpha_n\|f-g\|\big)\\
&= 4|a_1\Re a_n-b_1\Re b_n|^2
-24\alpha_n\|f-g\|^2\\
&= 4|a_1(\Re a_n-\Re b_n)+(a_1-b_1)\Re b_n|^2
-24\alpha_n\|f-g\|^2\\
&\geq 4\left(a_1^2(\Re a_n-\Re b_n)^2-2a_1|\Re a_n-\Re b_n||a_1-b_1||\Re b_n|\right)
-24\alpha_n\|f-g\|^2\\
&\geq |\Re a_n-\Re b_n|^2-2\alpha_n\|f-g\|^2
-24\alpha_n\|f-g\|^2\hspace{.7cm}\textrm{ as $1\geq a_1\geq .5$ and }1\geq b_1\geq 0,\\
&= |\Re a_n-\Re b_n|^2
-26\alpha_n\|f-g\|^2.
\end{align*}
The same argument with $\varphi_{n,i}$ instead of $\varphi_{n,1}$ gives that 
$$16\big||\langle f,\varphi_{n,i}\rangle|-|\langle g,\varphi_{n,i}\rangle|\big|^2
\geq |\Im a_n-\Im b_n|^2-26\alpha_n\|f-g\|^2.
$$
Hence, 
\begin{align*}
\|\mathcal{A}_\Phi(f)-\mathcal{A}_\Phi(g)\|^2&=\big||\langle f,e_{1}\rangle|-|\langle g,e_{1}\rangle|\big|^2+\sum_{n=2}^\infty\left( \big||\langle f,\varphi_{n,1}\rangle|-|\langle g,\varphi_{n,1}\rangle|\big|^2+ \big||\langle f,\varphi_{n,i}\rangle|-|\langle g,\varphi_{n,i}\rangle|\big|^2\right)\\
&\geq|a_1-b_1|^2+16^{-1}\sum_{n=2}^\infty|\Re a_n-\Re b_n|^2+|\Im a_n-\Im b_n|^2-52\alpha_n\|f-g\|^2\\
&= |a_1-b_1|^2+16^{-1}\sum_{n=2}^\infty |a_n-b_n|^2-52\alpha_n\|f-g\|^2\\
&\geq 16^{-1}\|f-g\|^2-(52/16)\sum_{n=2}^\infty \alpha_n\|f-g\|^2\\
&\geq 25^{-1}\|f-g\|^2\hspace{.5cm}\textrm{ as }\sum_{n=2}^\infty \alpha_n<1/200.
\end{align*}
This completes the proof.
\end{proof}

\section{Multidimensional examples}\label{S:M}

Note that the instabilities that occur in \eqref{gliding hump} are obtained by ``pushing bumps to infinity". In many practical situations, such instabilities would be considered to be unphysical. For example, when recording data from a power spectrum, the spectrum will be highly localized to the compact region of space where the experiment is taking place, with tails that are rapidly decaying away from this region. Experimentally, one is primarily interested in recovering the main localized piece, so instabilities caused by the decaying tails should not be considered as a fundamental obstruction to recovering the signal.
\medskip

  To model such scenarios, we consider a mathematical setting that is very similar to the above. We fix a filtration $V_1\subsetneq V_2\subsetneq\ldots \subsetneq \mathcal{H}$ of an infinite-dimensional Hilbert space $\mathcal{H}$ by finite-dimensional linear subspaces $(V_m)_{m=1}^\infty$ with $V_1$ doing stable phase retrieval. We then define $\mathcal{B}$ to be the set of all vectors $f\in \mathcal{H}$ for which $\|f-P_mf\|_{\mathcal{H}}$ goes to zero at a sufficiently fast rate. Here, we view $V_1$ as the space where the bulk of the physically relevant signals are localized, with the rest of the Hilbert space being used to model perturbations of this main finite-dimensional piece. By choosing the rates appropriately, the enforced decay of the tails will prevent the image of the signals in $\mathcal{B}$ from migrating towards infinity as in \eqref{gliding hump}, making stable phase retrieval conceivable for these models.  In Theorems \ref{T:R} and \ref{T:C} we obtained stability for phase retrieval by forcing the set of vectors $\mathcal{B}$ to be highly concentrated around a fixed one-dimensional subspace.  Of course, in applications the subspace $V_1$ could have arbitrarily large dimension.  In the remainder of this section, we will prove that stable phase retrieval for such sets $\mathcal{B}$ is still possible for $V_1$ having any finite dimension.

\begin{theorem}\label{T:RMD}
Let $V_1$ be a finite-dimensional subspace of a separable infinite-dimensional real Hilbert space $\mathcal{H}$ and let $(e_n)_{n=2}^\infty$ be an orthonormal basis of $V_1^\perp$. For each $m\in\N$, let $V_m=V_1+\textrm{span}_{2\leq n\leq m}e_n$ and let $P_m:\mathcal{H}\rightarrow V_m$ be the orthogonal projection.
Let $(\psi_j)_{j\in I}\subseteq V_1$ be a frame of $V_1$ with frame bounds $0<A\leq B\leq 1$ which does $C$-stable phase retrieval on $V_1$.  Let $(\varphi_j)_{j\in J}\subseteq V_1$ be a frame of $V_1$ with frame bounds $0<A\leq B\leq 1$ for which there exists $c>0$ so that, for all $x\in V_1$, if $J_c(x)=\{j\in J: |\langle x,\varphi_j\rangle|\geq c\|x\|\}$ then $|J_c|\geq C^{-1}|J|$.
Fix $0<\vp<8^{-1}$ and choose $\alpha_n\searrow 0$ such that $\sum_{n=2}^\infty \alpha_n <2^{-2}C^{-1}|J|^{-1/2}\vp$ and $\beta_n\searrow 0$ such that $\beta_n<2^{-6}c^2\alpha_n$ for all $n\geq 2$.   Define the subset
\begin{equation*}
\mathcal{B}=\{f\in \mathcal{H} : \|f-P_mf\|\leq \beta_{m+1}\|f\|\ \text{for every}\ m\in \mathbb{N}\}
\end{equation*} 
 and consider the frame $\Phi:=(\psi_{j})_{j\in I}\cup(\varphi_{j,n})_{j\in J, n\geq 2}$ of $\mathcal{H}$ where $\varphi_{j,n}:=\alpha_{n} \varphi_j+|J|^{-1/2}e_n$ for all $j\in J$ and $n\geq 2$. Then $\Phi$ has lower frame bound $(1-\sqrt{\frac{\vp}{A}})^2A$ and upper frame bound $(1+\sqrt{\vp})^2$.  Furthermore,
for all $f,g\in \mathcal{B}$ we have
\begin{equation}\label{Conclusion}
\inf_{|\alpha|=1}\|f-\alpha g\|^2\leq (1-\vp)^{-1}C^2\|\mathcal{A}_\Phi(f)-\mathcal{A}_\Phi(g)\|^2.
\end{equation}
\end{theorem}
\begin{remark}
Note that the parameter $c$ does not appear quantitatively in the conclusion \eqref{Conclusion}. Instead, it is used to enforce the rapid convergence of $\beta_n$ relative to $\alpha_n$, which in turn guarantees that $\sign \langle f,\varphi_j\rangle=\sign\langle f,\varphi_{j,n}\rangle$.
\end{remark}
\begin{proof}
Let $f,g\in\mathcal{B}$ with $\|f\|=1$ and $\|g\|\leq 1$. Without loss of generality, we may assume that $\langle P_1 f,P_1 g\rangle\geq0$.
\medskip

We first consider the case that $\|P_1 f- P_1 g\|^2\geq \vp$.
In this case, as $(\psi_j)_{j\in J}$ does $C$-stable phase retrieval on $V_1$, we have that
\begin{align*}
\|f-g\|^2&=\|P_{V_1} f-P_{V_1} g\|^2+\|P_{V^\perp_1} f-P_{V^\perp_1} g\|\\
&<\|P_{V_1} f-P_{V_1} g\|^2+4\beta_2^2\\
&<\|P_{V_1} f-P_{V_1} g\|^2+\vp^2\\
&<(1+\vp)\|P_{V_1} f-P_{V_1} g\|^2\hspace{1cm}\textrm{ as }\|P_{V_1} f-P_{V_1} g\|^2>\vp,\\
&\leq (1-\vp)^{-1}C^2\|\mathcal{A}_\Phi(f)-\mathcal{A}_\Phi(g)\|^2.
\end{align*}

We now  consider the case that $\|P_1 f- P_1 g\|^2\leq \vp$.  Let $j\in J_{c}(P_1 f)$ and $n\geq 2$.  We have that $|\langle P_1 f,\varphi_j\rangle|\geq c\|P_1 f\|\geq c(1-\beta_2)$.  Note that $\|\varphi_j\|\leq 1$ as $(\varphi_j)_{j\in J}$ has upper frame bound $B\leq 1$.  As $\vp^{1/2}<c(1-\beta_2)$, this gives that $|\langle P_1 g,\varphi_j\rangle|\geq c(1-\beta_2)-\vp^{1/2}$ and $\text{sign}(\langle P_1 f,\varphi_j\rangle)=\text{sign}(\langle P_1 g,\varphi_j\rangle)$. Consider now the expansions $$P_{V_1^\perp}f=\sum_{n=2}^\infty a_n e_n \ \ \text{and}\ \  P_{V_1^\perp}g=\sum_{n=2}^\infty b_n e_n.$$ 
This gives that
\begin{align*}
|\langle f,\varphi_{j,n}\rangle-\alpha_n\langle P_1 f,\varphi_{j}\rangle|&=|\langle f,\alpha_n\varphi_{j}+|J|^{-1/2}e_n\rangle-\alpha_n\langle P_1 f,\varphi_{j}\rangle|\\
&=|J|^{-1/2} |a_n|\\
&\leq |J|^{-1/2} \beta_{n+1}\\
&\leq |J|^{-1/2} \beta_{n+1}(1-\beta_2)^{-1}c^{-1}|\langle P_1 f,\varphi_j\rangle|\\ 
&<2^{-1}\alpha_n|\langle P_1 f,\varphi_j\rangle|.
\end{align*}
Thus, $\text{sign}(\langle f,\varphi_{j,n}\rangle)=\text{sign}(\alpha_n\langle P_1 f,\varphi_{j}\rangle)$.  Likewise,
\begin{align*}
|\langle g,\varphi_{j,n}\rangle-\alpha_n\langle P_1 g,\varphi_{j}\rangle|&=|J|^{-1/2} |b_n|\\
&\leq |J|^{-1/2} \beta_{n+1}\\ 
&\leq |J|^{-1/2} \beta_{n+1}(c(1-\beta_2)-\vp^{1/2})^{-1}|\langle P_1 f,\varphi_j\rangle|\\ 
&<2^{-1}\alpha_n|\langle P_1 f,\varphi_j\rangle|.
\end{align*}
Thus, sign$(\langle g,\varphi_{j,n}\rangle)=\text{sign}(\alpha_n\langle P_1 g,\varphi_{j}\rangle)$.  As $\text{sign}(\langle P_1 f,\varphi_j\rangle)=\text{sign}(\langle P_1 g,\varphi_j\rangle)$, we also have that $\text{sign}(\langle f,\varphi_{j,n}\rangle)=\text{sign}(\langle g,\varphi_{j,n}\rangle)$. This gives the key property that $$
||\langle f,\varphi_{j,n}\rangle|-|\langle g,\varphi_{j,n}\rangle||=|\langle f,\varphi_{j,n}\rangle-\langle g,\varphi_{j,n}\rangle|.
$$
We now compute
\begin{align*}
\|\mathcal{A}_\Phi(f)&-\mathcal{A}_\Phi(g)\|^2=\sum_{j\in J} \big||\langle f,\psi_j\rangle|-|\langle g,\psi_j\rangle|\big|^2+\sum_{n=2}^\infty\sum_{j\in J} \big||\langle f,\varphi_{j,n}\rangle|-|\langle g,\varphi_{j,n}\rangle|\big|^2\\
&\geq C^{-2}\|P_1 (f-g)\|^2+\sum_{n=2}^\infty\sum_{j\in J} \big||\langle f,\varphi_{j,n}\rangle|-|\langle g,\varphi_{j,n}\rangle|\big|^2\\
&\geq C^{-2}\|P_1 (f-g)\|^2+\sum_{n=2}^\infty\sum_{j\in J_c(P_1 f)} \big|\langle f,\varphi_{j,n}\rangle-\langle g,\varphi_{j,n}\rangle\big|^2\\
&= C^{-2}\|P_1 (f-g)\|^2+\sum_{n=2}^\infty\sum_{j\in J_c(P_1 f)} \big|\alpha_n\langle P_1 (f-g),\varphi_j\rangle+(a_n -b_n) |J|^{-1/2}\big|^2\\
&\geq C^{-2}\|P_1 (f-g)\|^2+\sum_{n=2}^\infty\sum_{j\in J_c(P_1 f)} \big||a_n -b_n| |J|^{-1/2}-\alpha_n|\langle P_1 (f-g),\varphi_j\rangle|\big|^2\\
&\geq C^{-2}\|P_1 (f-g)\|^2+\sum_{n=2}^\infty\sum_{j\in J_c(P_1 f)} \big||a_n -b_n|^2 |J|^{-1}-2|a_n -b_n| |J|^{-1/2}\alpha_n|\langle P_1 (f-g),\varphi_j\rangle|\big|\\
&\geq C^{-2}\|P_1 (f-g)\|^2+\sum_{n=2}^\infty |a_n -b_n|^2 |J_c(P_1 f)||J|^{-1}-2|a_n -b_n| |J|^{1/2}\alpha_n\| P_1 (f-g)\|\\
&\geq C^{-2}\|P_1 (f-g)\|^2+\sum_{n=2}^\infty C^{-2}|a_n -b_n|^2 -2|J|^{1/2}\alpha_n\| f-g\|^2\\
&= C^{-2}\|f-g\|^2-\sum_{n=2}^\infty 2|J|^{1/2}\alpha_n\| f-g\|^2\\
&\geq C^{-2}\|f-g\|^2-\varepsilon C^{-2}\| f-g\|^2.
\end{align*}

Thus, we have proven \eqref{Conclusion}. Recall that
 $\Phi:=(\psi_{j})_{j\in I}\cup(\varphi_{j,n})_{j\in J, n\geq 2}$  where $\varphi_{j,n}:=\alpha_{n} \varphi_j+|J|^{-1/2}e_n$ for all $j\in J$ and $n\geq 2$. As $(\psi_{j})_{j\in I}$ is a frame of $V_1$ with frame bounds $A\leq B\leq 1$, it follows that 
 $(\psi_{j})_{j\in I}\cup(|J|^{-1/2}e_n)_{j\in J, n\geq 2}$
 has lower frame bound $A$ and upper frame bound $1$. Thus, 
 as $\Phi$ is an $\varepsilon$-perturbation of $(\psi_{j})_{j\in I}\cup(|J|^{-1/2}e_n)_{j\in J, n\geq 2}$, we have that  $\Phi$
 has lower frame bound $(1-\sqrt{\frac{\vp}{A}})^2A$ and upper frame bound $(1+\sqrt{\vp})^2$ by \cite{Christensen}. 
\end{proof}

 Before stating the theorem in the complex case, we set some general conditions. Let $\mathcal{H}$ be a complex Hilbert space and let $V_1\subseteq \mathcal{H}$ be a finite-dimensional subspace with $N=\dim(V_1)$.  We will use a frame $(\varphi_j)_{j\in J}$ of $V_1$, constants $c,\kappa,\vp>0$ and  sequences $(\alpha_n)_{n=1}^\infty,(\beta_n)_{n=1}^\infty$ such that the following conditions hold.
\begin{enumerate}
\item For all $x,y\in V_1$, 
$$|J|\leq \kappa \big|\big\{j\in J: c^{-1}\|x\|\geq \sqrt{N}|\langle x,\varphi_j\rangle|\geq c\|x\|\textrm{ and } c^{-1}\|y\|\geq \sqrt{N}|\langle y,\varphi_j\rangle|\geq c\|y\|\big\}\big|,$$
\item $\sum_{n\geq 2}\alpha_n^2\leq \vp^2 11^{-1}|J|^{-1}c^{-2}\min( C^{-1}, \kappa^{-1}64^{-1}c^2),$
\item $\beta_n<2^{-6}c^2 N^{-1/2} \alpha_n$ for all $n\geq 2$.
\end{enumerate}

\begin{theorem}\label{T:CMD}
Let $V_1$ be a finite-dimensional subspace of a separable infinite-dimensional complex Hilbert space $\mathcal{H}$ and let $(e_n)_{n=2}^\infty$ be an orthonormal basis of $V_1^\perp$. For each $m\in\N$, let $V_m=V_1+\textrm{span}_{2\leq n\leq m}e_n$ and let $P_m:\mathcal{H}\rightarrow V_m$ be the orthogonal projection.
Let $A,B,C>0$ and suppose that $(\psi_j)_{j\in I}$ is a frame of $V_1$ with lower frame bound $A$, upper frame bound $B$, and such that 
$$\inf_{|\alpha|=1}\|x-\alpha y\|^2\leq C\sum_{j\in J_c}\big||\langle x,\psi_j\rangle|-|\langle y,\psi_j\rangle|\big|^2\hspace{.5cm}\textrm{ for all }x,y\in V_1.
$$
Let $N=dim(V_1)$ and  let $(\varphi_{j})_{j\in J}\subseteq V_1$ be a frame of $V_1$ with frame bounds $0<A\leq B\leq 1$ and  let $c,\kappa>0$ and  $0<\vp<\min(8^{-1},A)$ be constants which satisfy the above conditions (i),(ii),(iii).
Define the subset
\begin{equation*}
\mathcal{B}=\{f\in \mathcal{H} : \|f-P_mf\|\leq \beta_{m+1}\|f\|\ \text{for every}\ m\in \mathbb{N}\}.
\end{equation*} 
 Let $\Phi:=(\psi_{j})_{j\in I}\cup(\varphi_{j,n,1})_{j\in J, n\geq 2}\cup(\varphi_{j,n,i})_{j\in J, n\geq 2}$ be the frame of $\mathcal{H}$ where for $j\in J$ and $n\geq 2$,
 $$\varphi_{j,n,1}:=\alpha_{n} \varphi_j+(2|J|)^{-1/2}e_n\textrm{ and }\varphi_{j,n,i}:=\alpha_{n} \varphi_j-(2|J|)^{-1/2}ie_n.$$
 Then 
for all $f,g\in \mathcal{B}$ we have
\begin{equation*}
\inf_{|\alpha|=1}\|f-\alpha g\|^2\leq (1-\vp)^{-1}\max( C, 64\kappa c^{-2})\|\mathcal{A}_\Phi(f)-\mathcal{A}_\Phi(g)\|^2.
\end{equation*}
Furthermore, $\Phi$ has lower frame bound has lower frame bound $(1-\sqrt{\frac{\vp}{A}})^2A$ and upper frame bound $(1+\sqrt{\vp})^2$.   
\end{theorem}
\begin{remark}
 Note that Frames $(\varphi_j)_{j\in J}$ verifying the hypotheses of \Cref{T:RMD} are abundant and one may compare the conditions in \Cref{T:CMD} with \cite[Theorem 5.6]{FOTP}. Further, as far was we know, every frame of random vectors that has been proven to do uniformly stable phase retrieval with high probability will also satisfy the hypotheses of \Cref{T:RMD} as for example in \cite{MR3260258,MR3069958,MR3175089,MR3746047,KS}.
\end{remark}
\begin{proof}
For simplicity of notation, we let $K=2|J|$.  
Let $f,g\in\mathcal{B}$ with $\|f\|=1$ and $\|g\|\leq 1$. For simplicity, we denote $a_n=\langle f,e_n\rangle$ and $b_n=\langle g,e_n\rangle$ for all $n\geq2$.
 We may assume without loss of generality that $\langle P_{V_1} f, P_{V_1}g\rangle\geq 0$ and $\langle f,\varphi_j\rangle>0$ for all $j\in J$. 
 This implies that
\begin{equation*}\label{E:fgC}
\|P_{V_1}f-P_{V_1}g\|^2=\inf_{|\alpha|=1}\|P_{V_1}f-\alpha P_{V_1}g\|^2\leq C
\sum_{j\in I}\big||\langle P_{V_1}f,\psi_j\rangle|-|\langle P_{V_1}g,\psi_j\rangle|\big|^2.
\end{equation*}
We denote $J_c$ as the set
$$J_c=\big\{j\in J: c^{-1}\geq \sqrt{N}|\langle \|P_1 f\|^{-1}f,\varphi_j\rangle|\geq c\textrm{ and } c^{-1}\geq \sqrt{N}|\langle \|P_1 g\|g,\varphi_j\rangle|\geq c\big\}.$$

For each $j\in J_c$ we denote $\delta_{j}=\textrm{phase} \overline{\langle g,\varphi_j\rangle}$.   That is, $\delta_j\langle g,\varphi_j\rangle=|\langle g,\varphi_j\rangle|$.  
We now consider a fixed $j\in J_c$ and $n\geq 2$.
It is easy to see that
\begin{align*}
|\langle f,\varphi_{j,n,1}\rangle|^2&=|\alpha_n \langle f,\varphi_j\rangle +K^{-1/2}a_n|^2\\
&=(\alpha_n \langle f,\varphi_j\rangle+K^{-1/2}\Re a_n)^2+(K^{-1/2}\Im a_n)^2\\
&=\alpha^2_n \langle f,\varphi_j\rangle^2+2K^{-1/2}\alpha_n \langle f,\varphi_j\rangle\Re a_n+K^{-1}(\Re a_n)^2+K^{-1}(\Im a_n)^2\\
&=\alpha^2_n \langle f,\varphi_j\rangle^2+2K^{-1/2}\alpha_n \langle f,\varphi_j\rangle\Re a_n+K^{-1}|a_n|^2.
\end{align*}
Likewise, 
\begin{align*}
|\langle f,\varphi_{j,n,i}\rangle|^2&=|\alpha_n \langle f,\varphi_j\rangle -K^{-1/2}a_n i|^2\\
&=(\alpha_n \langle f,\varphi_j\rangle+K^{-1/2}\Im a_n)^2+(K^{-1/2}\Re a_n)^2\\
&=\alpha^2_n \langle f,\varphi_j\rangle^2+2K^{-1/2}\alpha_n \langle f,\varphi_j\rangle\Im a_n+K^{-1}(\Im a_n)^2+K^{-1}(\Re a_n)^2\\
&=\alpha^2_n \langle f,\varphi_j\rangle^2+2K^{-1/2}\alpha_n \langle f,\varphi_j\rangle\Im a_n+K^{-1}|a_n|^2.
\end{align*}
Similarly, we have for $g$ that
\begin{align*}
|\langle g,\varphi_{j,n,1}\rangle|^2
&=|\alpha_n \langle g,\varphi_j\rangle +K^{-1/2} b_n|^2\\
&=|\alpha_n \delta_j \langle g,\varphi_j\rangle +K^{-1/2}\delta_j b_n|^2\\
&=\alpha^2_n |\langle g,\varphi_j\rangle|^2+2K^{-1/2}\alpha_n |\langle g,\varphi_j\rangle|\Re (\delta_j b_n)+K^{-1}|b_n|^2
\end{align*}
and
$$
|\langle g,\varphi_{j,n,i}\rangle|^2
=\alpha^2_n |\langle g,\varphi_j\rangle|^2+2K^{-1/2}\alpha_n |\langle g,\varphi_j\rangle|\Im(\delta_j b_n)+K^{-1}|b_n|^2.
$$
For the case of $\varphi_{j,n,1}$, we observe the identity
\begin{equation}\label{E:diff2}
\begin{split}
\big||\langle f,\varphi_{j,n,1}\rangle|-|\langle g,\varphi_{j,n,1}\rangle|\big|&=\big|(\alpha^2_n |\langle f,\varphi_j\rangle|^2+2\alpha_n K^{-1/2} |\langle f,\varphi_j\rangle|\Re a_n+K^{-1}|a_n|^2)^{1/2}
\\
&-(\alpha^2_n |\langle g,\varphi_j\rangle|^2+2\alpha_n K^{-1/2}|\langle g,\varphi_j\rangle|\Re (\delta_jb_n)+K^{-1}|b_n|^2)^{1/2}\big|.
\end{split}
\end{equation}
We have that $c^{-1}N^{-1/2}\geq \langle \|P_1 f\|^{-1}f,\varphi_j\rangle\geq cN^{-1/2}$ and hence $c^{-1}N^{-1/2}\geq \langle f,\varphi_j\rangle\geq 2^{-1}cN^{-1/2}$. It follows by (iii) that $|a_n|\leq \beta_{n+1}\leq 2^{-1}cN^{-1/2}\alpha_n\leq \alpha_n\langle f,\varphi_j\rangle$.    Thus,  we also have 
$$2\alpha_n c^{-1}N^{-1/2}\geq (\alpha^2_n | \langle f,\varphi_j\rangle|^2+2\alpha_n | \langle f,\varphi_j\rangle|\Re a_n+|a_n|^2)^{1/2}.$$
Summing this inequality with the corresponding one for $g$ gives that
\begin{equation}\label{E:alpha2}
\begin{split}
4\alpha_n c^{-1}N^{-1/2}\geq  (&\alpha^2_n |\langle f,\varphi_j\rangle|^2+2\alpha_n K^{-1/2}|\langle f,\varphi_j\rangle|\Re a_n
+K^{-1}|a_n|^2)^{1/2}
\\
+(&\alpha^2_n |\langle g,\varphi_j\rangle|^2+2\alpha_n K^{-1/2} |\langle g,\varphi_j\rangle|\Re (\delta_j b_n)+K^{-1}|b_n|^2)^{1/2}.
\end{split}
\end{equation}
By multiplying \eqref{E:diff2} and \eqref{E:alpha2}, we obtain
\begin{equation*}
\begin{split}
4\alpha_n c^{-1}N^{-1/2}\big||\langle f,\varphi_{j,n,1}\rangle|-|\langle g,\varphi_{j,n,1}\rangle|\big|\geq \big|&\alpha^2_n |\langle f,\varphi_j\rangle|^2+2\alpha_n |\langle f,\varphi_j\rangle|\Re a_n+|a_n|^2
\\
-(&\alpha^2_n |\langle g,\varphi_j\rangle|^2+2\alpha_n |\langle g,\varphi_j\rangle|\Re (\delta_j b_n)+|b_n|^2)\big|.
\end{split}
\end{equation*}
Thus, we have that
\begin{equation*}\label{E:1C1}
\begin{split}
16  c^{-2}N^{-1}\big||\langle f,\varphi_{j,n,1}\rangle|-|\langle g,\varphi_{j,n,1}\rangle|\big|^2&\geq \big|\alpha_n (|\langle f,\varphi_j\rangle|^2-|\langle g,\varphi_j\rangle|^2)+\alpha_n^{-1}K^{-1}(|a_n|^2-|b_n|^2)
\\
&+2K^{-1/2}(|\langle f,\varphi_j\rangle|\Re a_n-|\langle g,\varphi_j\rangle|\Re (\delta_jb_n))\big|^2.
\end{split}
\end{equation*}
Using the inequality of real numbers $|a+b+c|^2\geq 4^{-1}c^2-4b^2-4a^2$, we may bound
\begin{equation}\label{E:1C2}
\begin{split}
16c^{-2}N^{-1}\big||\langle f,\varphi_{j,n,1}\rangle|-|\langle g,\varphi_{j,n,1}\rangle|\big|^2&\geq 
K^{-1}\big||\langle f,\varphi_j\rangle|\Re a_n-|\langle g,\varphi_j\rangle|\Re (\delta_jb_n)\big|^2
\\
&-4\alpha_n^2 \big||\langle f,\varphi_j\rangle|^2-|\langle g,\varphi_j\rangle|^2\big|^2-4\alpha_n^{-2}K^{-2}\big||a_n|^2-|b_n|^2\big|^2.
\end{split}
\end{equation}
Note that $cN^{-1/2}\leq |\langle f,\varphi_j\rangle|,|\langle g,\varphi_j\rangle|\leq c^{-1}N^{-1/2}$ and $0\leq |a_n|,|b_n|\leq 2^{-1}\alpha_n^2$.
This allows us to derive the following inequalities:
\begin{equation}\label{E:21}
\big||\langle f,\varphi_j\rangle|^2-|\langle g,\varphi_j\rangle|^2\big|^2=(|\langle f,\varphi_j\rangle|+|\langle g,\varphi_j\rangle|)^2\big||\langle f,\varphi_j\rangle|-|\langle g,\varphi_j\rangle|\big|^2\leq 4c^{-2}N^{-1}\|f-g\|^2,
\end{equation}
\begin{equation}\label{E:31}
\big||a_n|^2-|b_n|^2\big|^2=(|a_n|+|b_n|)^2\big||a_n|-|b_n|\big|^2\leq \alpha_n^4\|f-g\|^2.
\end{equation}
By \eqref{E:1C2}, \eqref{E:21} and \eqref{E:31}  we deduce that 
\begin{align*}
16c^{-2}N^{-1}\big||\langle f,\varphi_{j,n,1}\rangle|&-|\langle g,\varphi_{j,n,1}\rangle|\big|^2\\
&\geq K^{-1}\big||\langle f,\varphi_j\rangle|\Re a_n-|\langle g,\varphi_j\rangle|\Re (\delta_j b_n)\big|^2-16\alpha_n^2\|f-g\|^2(4c^{-2}N^{-1}+K^{-2})\\
&\geq K^{-1}\big||\langle f,\varphi_j\rangle|\Re a_n-|\langle g,\varphi_j\rangle|\Re (\delta_j b_n)\big|^2
-16\alpha_n^2\|f-g\|^2(4c^{-2}N^{-1}+c^{-2}N^{-1})\\
&= K^{-1}\big|\Re(\langle f,\varphi_j\rangle a_n)-\Re(\overline{\langle g,\varphi_j\rangle} b_n)\big|^2
-80c^2N^{-1}\alpha_n^2\|f-g\|^2\\
&= K^{-1}\big|\Re(\langle f,\varphi_j\rangle a_n-\overline{\langle g,\varphi_j\rangle} b_n)\big|^2
-80c^2N^{-1}\alpha_n^2\|f-g\|^2.
\end{align*}
The same argument with $\varphi_{j,n,i}$ instead of $\varphi_{j,n,1}$ leads to the inequality
$$16c^{-2}N^{-1}\big||\langle f,\varphi_{j,n,i}\rangle|-|\langle g,\varphi_{j,n,i}\rangle|\big|^2
\geq K^{-1}\big|\Im(\langle f,\varphi_j\rangle a_n-\overline{\langle g,\varphi_j\rangle} b_n)\big|^2
-80c^2N^{-1}\alpha_n^2\|f-g\|^2.
$$
Summing the above inequalities gives that
\begin{align*}
16c^{-2}N^{-1}&\big(\big||\langle f,\varphi_{j,n,1}\rangle|-|\langle g,\varphi_{j,n,1}\rangle|\big|^2+\big||\langle f,\varphi_{j,n,i}\rangle|-|\langle g,\varphi_{j,n,i}\rangle|\big|^2)\\
&\geq K^{-1}\big|\langle f,\varphi_j\rangle a_n-\overline{\langle g,\varphi_j\rangle} b_n\big|^2
-160c^2N^{-1}\alpha_n^2\|f-g\|^2\\
&= K^{-1}\big|\langle f,\varphi_j\rangle (a_n-b_n)+ (\langle f,\varphi_j\rangle-\overline{\langle g,\varphi_j\rangle}) b_n\big|^2
-160c^2N^{-1}\alpha_n^2\|f-g\|^2\\
&\geq K^{-1}4^{-1}|\langle f,\varphi_j\rangle|^2 |a_n-b_n|^2- K^{-1}|\langle f,\varphi_j\rangle-\overline{\langle g,\varphi_j\rangle}|^2 |b_n|^2
-160c^2N^{-1}\alpha_n^2\|f-g\|^2\\
&\geq K^{-1}4^{-1}c^2 N^{-1} |a_n-b_n|^2- |\langle f,\varphi_j\rangle-\langle g,\varphi_j\rangle|^2 \beta_n^2
-160c^2N^{-1}\alpha_n^2\|f-g\|^2\\
&\geq K^{-1}4^{-1}c^2 N^{-1}|a_n-b_n|^2- \|f-g\|^2\beta_n^2
-160c^2N^{-1}\alpha_n^2\|f-g\|^2\\
&\geq K^{-1}4^{-1}c^2 N^{-1}|a_n-b_n|^2-161 c^2N^{-1}\alpha_n^2\|f-g\|^2.
\end{align*}

Simplifying the above inequality yields,
$$\big||\langle f,\varphi_{j,n,1}\rangle|-|\langle g,\varphi_{j,n,1}\rangle|\big|^2+\big||\langle f,\varphi_{j,n,i}\rangle|-|\langle g,\varphi_{j,n,i}\rangle|\big|^2\geq K^{-1}64^{-1}c^4|a_n-b_n|^2-11 c^4 \alpha_n^2 \|f-g\|^2.
$$

Hence, 
\begin{align*}
\|\mathcal{A}_\Phi(f)-\mathcal{A}_\Phi(g)\|^2&=\sum_{j\in I}\big||\langle f,\psi_{j}\rangle|-|\langle g,\psi_j\rangle|\big|^2
\\
&\hspace{.7cm}+\sum_{j\in J}\sum_{n=2}^\infty\left( \big||\langle f,\varphi_{j,n,1}\rangle|-|\langle g,\varphi_{j,n,1}\rangle|\big|^2+ \big||\langle f,\varphi_{n,i}\rangle|-|\langle g,\varphi_{n,i}\rangle|\big|^2\right)\\
&\geq C^{-1}\|P_{V_1}f-P_{V_1}g\|^2+\sum_{j\in J_c}\sum_{n=2}^\infty K^{-1}64^{-1}c^4|a_n-b_n|^2-11 c^4 \alpha_n^2 \|f-g\|^2\\
&\geq C^{-1}\|P_{V_1}f-P_{V_1}g\|^2+|J_c||J^{-1}64^{-1}c^2\|P_{V_1^\perp}f-P_{V_1^\perp}g\|^2-11|J_c|c^2(\sum_{n\geq 2}\alpha_n^2)\|f-g\|^2\\
&\geq \min(C^{-1}, \kappa^{-1}64^{-1}c^2)\|f-g\|^2-11|J_c|c^2(\sum_{n\geq 2}\alpha_n^2)\|f-g\|^2\\
&\geq (1-\vp)\min(C^{-1}, \kappa^{-1}64^{-1}c^2)\|f-g\|^2.
\end{align*}
Note that the last inequality is obtained as $(\alpha_n)_{n=2}^\infty$ satisfies $\sum_{n\geq 2}\alpha_n^2\leq \vp\min( C^{-1}, \kappa^{-1}64^{-1}c^2)11^{-1}K^{-1}c^{-2}$.

Lastly, as in Theorem \ref{T:RMD} we have that $\Phi$ is an $\varepsilon$ perturbation of a frame with lower frame bound $A$ and upper frame bound  $1$.  Thus,  $\Phi$
 has lower frame bound $(1-\sqrt{\frac{\vp}{A}})^2A$ and upper frame bound $(1+\sqrt{\vp})^2$
\end{proof}

\subsection{Acknowledgments} We would like to thank Jaume de Dios Pont and Palina Salanevich for discussions related to the physical applications of the above phase retrieval models which then inspired the results in \Cref{S:M}.  

\bibliographystyle{plain}
\bibliography{refs.bib}

\end{document}